\documentclass{article}
\usepackage{amssymb,amsmath,amsthm,graphicx}

\def\qed{\hfill {\hbox{${\vcenter{\vbox{               
   \hrule height 0.4pt\hbox{\vrule width 0.4pt height 6pt
   \kern5pt\vrule width 0.4pt}\hrule height 0.4pt}}}$}}}

\def\utr{\, \underline{\triangleright}\, }
\def\otr{\, \overline{\triangleright}\, }

\newtheorem{theorem}{Theorem}

\theoremstyle{definition}
\newtheorem{example}{Example}
\newtheorem{definition}{Definition}

\date{}

\title{\Large \textbf{Biquandle Fares and Link Invariants}}

\author{Sam Nelson\footnote{Email: Sam.Nelson@cmc.edu. Partially supported by Simons Foundation collaboration grant 702597.}\and
Stella Shah\footnote{Email: sshah0159@scrippscollege.edu}}

\begin{document}
\maketitle

\begin{abstract}
We introduce a new family of invariants of oriented classical and virtual 
knots and links using \textit{fares}, maps from paths in biquandle-colored
diagrams to an abelian coefficient group. We consider the cases of 1-fares 
and 2-fares, provide examples to show that the enhancements are proper 
and end with some open questions about the cases of $n$-fares for $n>2$.
\end{abstract}

\parbox{5.5in} {\textsc{Keywords:} biquandle fares, biquandles, quivers, link 
invariants

\smallskip

}

\section{\large\textbf{Introduction}}\label{I}

\textit{Biquandles} are algebraic structures providing solutions to the
set-theoretic Yang-Baxter equation and invariants of oriented knots and links.
More precisely, every oriented classical or virtual knot or link
has a \textit{fundamental biquandle} whose isomorphism class is a strong
invariant and whose homsets into finite biquandles provide an infinite family
of computable invariants. These homsets can be represented (and computed) via
\textit{biquandle colorings} of diagrams analogous to matrices representing 
homomorphisms in a module category.

The \textit{biquandle counting invariant} is the cardinality of this homset.
An \textit{enhancement} of the biquandle counting invariant is a generally
stronger invariant which determines the homset cardinality. In much previous 
work, enhancements have been defined by finding invariants of homset
elements represented as coloring of diagrams and collecting the values of these
invariants over the homset; infinite families of such examples include the
biquandle 2-cocycle invariants \cite{CJKLS,CEGN}, biquandle bracket invariants 
\cite{NOR,GNO}, biquandle coloring quiver enhancements \cite{FN,JN} and 
many more \cite{NW}.

In this paper we introduce a new infinite family of enhancements of the 
biquandle homset invariant known as \textit{biquandle fares}, inspired by 
the path algebra of a quiver (and the Tokyo subway system). A fare of order
$n$ is a map from $n$-tuples of elements of $X$ representing paths of length
$n$ in a biquandle coloring of an oriented knot or link diagram interpreted
as a directed graph to an abelian group $A$. Summing the fare contributions
over an appropriate set of paths yields a value in $A$; the fare axioms are 
then the conditions required to ensure that this sum in not changed by 
$X$-colored Reidemeister moves. Collecting the total fares over the homset 
yields the enhanced invariant in multiset form, which can be represented as
a polynomial in various ways.

The paper is 
organized as follows. In Section \ref{BH} we review the basics of biquandles 
and homsets. In Section \ref{BF} we introduce the biquandle fare idea and 
illustrate with some 
toy examples. We handle the cases of fares of order $1,2$ and $3$ and discuss 
the case of higher order fares. For fares of order 2 and greater, we consider 
three basic types of fares: complete fares, through fares and crooked fares.
We compute tables of invariant values for some examples choices of biquandle,
coefficient ring and fare table. We conclude in Section \ref{Q} with some 
questions for future research.

This paper, including all text, illustrations, and python code for 
computations, was written strictly by the authors without the use of 
generative AI in any form.

\section{\large\textbf{Biquandles and Homsets}}\label{BH}

We begin by recalling a standard definition (see \cite{EN} for more).

\begin{definition}
A \textit{biquandle} is a set $X$ with two binary operations 
$\utr,\otr:X\times X\to X$ satisfying 
\begin{itemize}
\item[(i)] For all $x\in X$, $x\utr x=x\otr x$,
\item[(ii)] For all $y\in X$, the maps $\alpha_y,\beta_y:X\to X$ defined by
$\alpha_y(x)=x\otr y$, $\beta_y(x)=x\utr y$ and $S:X\times X\to X\times X$
defined by $S(x,y)=(y\otr x, x\utr y)$ are invertible, and
\item[(iii)] For all $x,y,z$, we have the \textit{exchange laws}
\[
\begin{array}{rcl}
(x\utr y)\utr (z\utr y) & = & (x\utr z)\utr (y\otr z) \\
(x\utr y)\otr (z\utr y) & = & (x\otr z)\utr (y\otr z) \\
(x\otr y)\otr (z\otr y) & = & (x\otr z)\otr (y\utr z) 
\end{array}.\]
\end{itemize}
A biquandle in which $x\otr y=x$ for all $x,y$ is called a \textit{quandle}.
\end{definition}

\begin{example}
Standard examples of biquandles include
\begin{itemize}
\item Any group has biquandle structures given by
\[x\utr y=y^{-n}xy^n \quad \mathrm{and}\quad x\otr y= x\]
for any integer $n$
\item Any group has biquandle structure given by
\[x\utr y=y^{-1}xy^{-1} \quad \mathrm{and}\quad x\otr y= x^{-1}\]
\item Any module over a commutative unital ring $R$ has biquandle
structures given by
\[x\utr y=tx+(s-t)y \quad \mathrm{and}\quad x\otr y= sx\]
for units $t,s\in R$.
\end{itemize}
\end{example}

The biquandle axioms are chosen so that for every \textit{$X$-coloring} of
an oriented classical or virtual knot or link diagram 
(i.e., assignment of elements of $X$ to the semiarcs of $D$ satisfying 
the condition
\[\scalebox{0.8}{\includegraphics{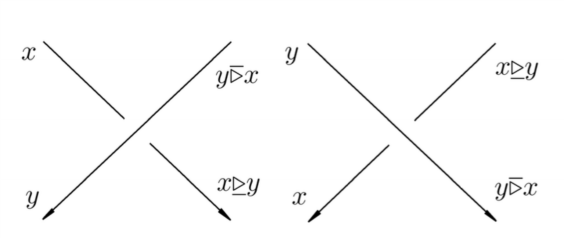}}\]
at every crossing) before a Reidemeister move, there is a unique $X$-coloring
of the diagram after the move which coincides with the original coloring
outside the neighborhood of the move.

An $X$-coloring determines and is determined by a unique biquandle 
homomorphism $f:\mathcal{B}(L)\to X$ from the \textit{fundamental biquandle}
of our oriented link $L$ to $X$, subject to a choice of generating set; see
\cite{EN,N25} for more. The set of $X$-colorings of a diagram of $L$ is called
the \textit{biquandle homset invariant} $\mathrm{Hom}(\mathcal{B}(L),X)$; 
its cardinality is known as the \textit{biquandle counting invariant}, denoted 
$\Phi_X^{\mathbb{Z}}(L)=|\mathrm{Hom}(\mathcal{B}(L),X)|$.

\begin{example}
Let $X$ be the biquandle given by the operation tables
\[
\begin{array}{r|rrr}
\utr & 1 & 2 & 3 \\ \hline
1 & 2 & 2 & 2 \\
2 & 1 & 1 & 1 \\
3 & 3 & 3 & 3 \\
\end{array}
\quad
\begin{array}{r|rrr}
\utr & 1 & 2 & 3 \\ \hline
1 & 2 & 3 & 1 \\
2 & 3 & 1 & 2 \\
3 & 1 & 2 & 3
\end{array}
\]
and consider the trefoil knot $3_1$. There are three $X$-colorings
of $3_1$, given by the homset
\[
\mathrm{Hom}(3_1,X)=\left\{
\raisebox{-0.55in}{\includegraphics{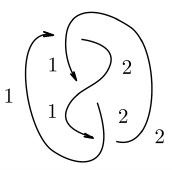}},\
\raisebox{-0.55in}{\includegraphics{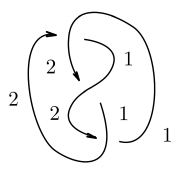}},\
\raisebox{-0.55in}{\includegraphics{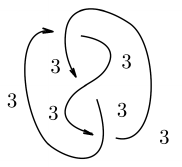}}
\right\}.
\]
Hence we have $\Phi_{X}^{\mathbb{Z}}(3_1)=3.$
\end{example}

\section{\large\textbf{Biquandle Fares}}\label{BF}

We begin with a definition.

\begin{definition}
Let $L$ be an oriented classical or virtual knot or link diagram.
A \textit{route} of order $n$ in $L$ is an ordered sequence of $n$
semiarcs in $L$ such that the terminal crossing point of the $k$th edge
is the initial crossing point of the $(k+1)$st edge for each $k=1,\dots, n-1$.
If we consider the underlying graph of the diagram $L$ as a quiver then
routes are the paths in the quiver.
\end{definition}

Let $X$ be a finite biquandle and $A$ and abelian group. 
A \textit{biquandle fare} of order $n$ with coefficients in $A$ is a linear map
$\phi:\mathbb{Z}[X^n]\to A$ satisfying axioms arising from the Reidemeister 
moves on $X$-colored diagrams where we interpret an $n$-tuple 
$(x_1,\dots,x_n)$ as corresponding to the biquandle colors on a route
of order $n$ in a biquandle coloring of a diagram $L$. More 
precisely, for each route of order $n$ we collect a fare contribution of 
$(-1)^k\phi(x_1,\dots,x_n)\in A$ where $k$ is the number of negative crossings 
internal to the route. 
\[\begin{array}{c}\includegraphics{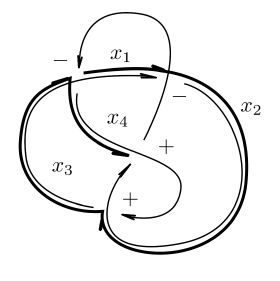} \\ 
+(-1)^2\phi(x_1,x_2,x_3,x_4)\end{array}\]

Given an $X$-colored
diagram, we find the set of all such contributions and take their sum in $A$ as 
the total fare of the diagram. The fare axioms are then chosen such that this 
sum of $\phi$-values in $A$ is invariant under $X$-colored Reidemeister moves 
and hence defines an enhancement of the biquandle homset invariant.

The simplest case is fares of order $n=1$.

\begin{definition}
Let $X$ be a finite biquandle and $A$ an abelian group. Then a 
\textit{biquandle fare} of order $n=1$ with coefficients in $A$ is a function
$\phi:X\to A$ satisfying the conditions
\begin{itemize}
\item[(i)] For all $x\in X$, $\phi(x\utr x)+\phi(x)=0$,
\item[(ii)] For all $x,y\in X$, $\phi(x)+\phi(y)+\phi(x\utr y)+\phi(y\otr x)=0$
and
\item[(iii)] For all $x,y,z\in X$, 
\[\phi(y)+\phi(x\utr y)+\phi(x\otr y)
=\phi(z\otr x)+\phi(x\utr z)+\phi((y\otr x)\utr(z\otr x))\]
\end{itemize}
\end{definition}

To compute the fare of an $X$-colored diagram we can simply replace the
biquandle colors $x$ on each edge with their values $\phi(x)$ in $A$ and sum
them up. To see that the above conditions give invariance under Reidemeister 
moves it suffices to check each move in a generating set of oriented 
Reidemeister moves. In \cite{P} several such sets are identified; we will use 
the (non-minimal) set consisting of all four moves of type I, all four moves of 
type II, and the all-positive move of type III.

\begin{definition}
Let $X$ be a finite biquandle, $D$ an oriented knot or link diagram representing
and oriented knot or link $L$, $A$ an abelian group and $\phi$ a biquandle 
fare. The multiset of of all appropriate fare values over the biquandle 
homset $\mathrm{Hom}(\mathcal{B}(L),X)$
is the \textit{biquandle fare multiset}, denoted $\Phi_X^{\phi,M}(L)$.
\end{definition}

\begin{theorem}\label{thm1}
Let $X$ be a biquandle, $L$ an oriented classical or virtual knot, link  
multi-knotoid, $A$ an abelian group and $\phi:X\to A$ a biquandle $1$-fare.
Then the multiset of fare values over the biquandle homset is invariant under
Reidemeister moves.
\end{theorem}

\begin{proof}
Taking Reidemeister I, we see that we need 
$\phi(x)+\phi(x)+\phi(x\utr x)=\phi(x)$ or more simply 
$\phi(x)+\phi(x\utr x)=0$:
\[\includegraphics{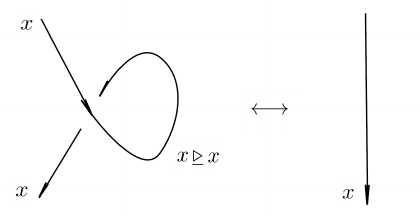}\]
The reader can check that the other three type I moves yield the same condition.

For Reidemeister II, we need 
\[\phi(x\utr y)+\phi(y)+\phi(x\utr y)+\phi(y\otr x)+\phi(x)+\phi(y)
=\phi(x\utr y)+\phi(y)\]
\[\includegraphics{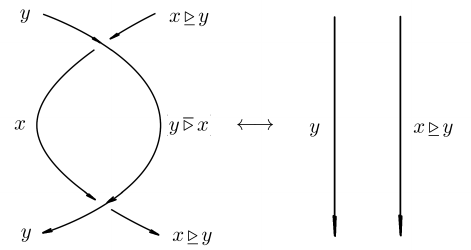}\]
which is equivalent to $\phi(x)+\phi(y)+\phi(x\utr y)+\phi(y\otr x)=0$.
As in the type I case, the reader can verify that all four type II moves
yield the same condition.

Finally, in the type III move we note that the colors on the outside of the
move neighborhood yield the same sum on both sides of the move by the biquandle
axioms, so we only need to consider the inner triangle contributions in the
$1$-fare case.
\[\includegraphics{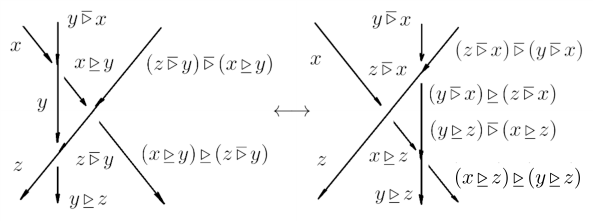}.\]
\end{proof}

We can express a 1-fare for a biquandle $X=\{x_1,\dots, x_n\}$
as an ordered $n$-tuple $\phi=(\phi(x_1),\dots, \phi(x_n))$.

\begin{example}
Let $X$ be the biquandle defined by the operation tables
\[
\begin{array}{r|rrr}
\utr & 1 & 2 & 3 \\ \hline
1 & 2 & 2 & 1 \\
2 & 1 & 1 & 2 \\
3 & 3 & 3 & 3
\end{array}
\quad \begin{array}{r|rrr}
\otr & 1 & 2 & 3 \\ \hline
1 & 2 & 2 & 2 \\
2 & 1 & 1 & 1 \\
3 & 3 & 3 & 3
\end{array}.
\]
Then we compute that $X$ has $1$-fares over $\mathbb{Z}_2$ given in the table
\[
\begin{array}{r|cccc}
x & \phi_1(x) & \phi_2(x) & \phi_3(x) & \phi_4(x) \\ \hline
1 & 0 & 0 & 1 & 1 \\
2 & 0 & 1 & 0 & 1 \\
3 & 0 & 1 & 0 & 1
\end{array}.
\]
\end{example}

We will denote the $1$-fare multiset invariant as $\Phi_{X}^{\phi,M}(L).$
We can optionally express the multiset in one of several ``polynomial'' 
formats (scare quotes because these are not necessarily polynomials in 
the usual sense, depending on the choice of abelian group $A$):

\begin{definition}
Let $X$ be a biquandle, $L$ an oriented classical of virtual knot or link,
$A$ and abelian group, $\phi:\mathbb{Z}[X^n]\to A$ a biquandle $n$-fare and
$\Phi_{X}^{\phi,M}(L)$ the $n$-fare multiset invariant. Then the 
\textit{additive $n$-fare polynomial} of $L$ with respect to $(X,A,\phi)$
is
\[\Phi_{X}^{\phi,+}(L)=\sum_{f\in\mathrm{Hom}(\mathcal{B}(L),X)} x^{\phi(f)}\]
and the 
\textit{multiplicative $n$-fare polynomial} of $L$ with respect to $(X,A,\phi)$
is
\[\Phi_{X}^{\phi,\times}(L)=\prod_{f\in\mathrm{Hom}(\mathcal{B}(L),X)} (x-\phi(f)).\]
\end{definition}

The two polynomial forms are equivalent over most rings. For the additive 
form, evaluation at $x=1$ yields the counting invariant value and
the coefficient of a monomial tells us the number of colorings whose
fare value is the monomial's exponent, while in the multiplicative form 
the counting invariant is the degree of the polynomial and the fares are 
its roots.

\begin{example}
Let $X$ be the biquandle defined by the operation tables
\[
\begin{array}{r|rrr}
\utr & 1 & 2 & 3 \\ \hline
1 & 3 & 1 & 3 \\
2 & 2 & 2 & 2 \\
3 & 1 & 3 & 1
\end{array}
\quad \begin{array}{r|rrr}
\otr & 1 & 2 & 3 \\ \hline
1 & 3 & 3 & 3 \\
2 & 2 & 2 & 2 \\
3 & 1 & 1 & 1 \\
\end{array}
\]
and let $A=\mathbb{Z}_2\oplus\mathbb{Z}_2$ be the Klein 4-group. Then
we compute that $X$ has 1-fares over $A$ including $\phi=((1,0),(0,1),(1,0))$.
Then the virtual Hopf link $L$ has three $X$-colorings
\[\includegraphics{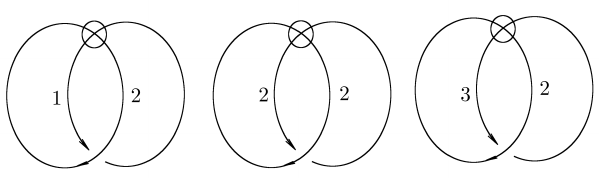}\]
Collecting the $1$-fares over this homset yields the $1$-fare multiset
invariant value $\Phi_X^{\phi,M}(L)=\{2\times (1,1), 1\times (0,0)\}$.
We note that this distinguishes the virtual Hopf link from the unlink $U$ since
a diagram without crossings has an empty set of routes and hence its multiset
$\Phi_{X}^{\phi,M}(U)$ can contain only zeros.
\end{example}

For $n>1$, several cases are possible; we will consider three.
For a \textit{complete fare} we collect 
contributions from all possible routes of length $n$. For a 
\textit{through fare} we collect contributions from routes of length $n$ 
continuing along the knot either over or under at each crossing. For a 
\textit{crooked fare} we collect contributions only from routes of length $n$
which turn at every vertex, never continuing across the crossing.

\begin{definition}
Let $X$ be a finite biquandle and $A$ an abelian group. Then a \textit{complete
biquandle fare} of order $n=2$ with coefficients in $A$ is a function
$\phi:X^2\to A$ satisfying the conditions
\begin{itemize}
\item[(i)] For all $x\in X$, 
\[\phi(x,x)+\phi(x,x\utr x)+\phi(x\utr x,x)+\phi(x\utr x,x\utr x)=0,\]
\item[(ii)] For all $x,y\in X$, 
\[\phi(x,y)+\phi(x,x\utr y)+\phi(y\otr x,y)+\phi(y\otr x,x\utr y)
-\phi(y,x)-\phi(x\utr y,x)-\phi(y,y\otr x)-\phi(x\utr y,y\otr x)=0\]
and
\item[(iii)] For all $x,y,z\in X$, 
\[\scalebox{0.95}{$\begin{array}{rcl}
\phi(x,y)
+\phi(x,x\utr y)
+\phi(y\otr x,y)
+\phi(y\otr x,x\utr y) 
& & \phi(x,z)+\phi(x,x\utr z)
+\phi(z\otr x,z)+\phi(z\otr x,x\utr z)\\
+\phi(y,z)+\phi(y,y\utr z)
+\phi(z\otr y,z)+\phi(z\otr y, y\utr z) 
& & +\phi(y\otr x,z\otr x)
+\phi(y\otr x,(y\otr x)\utr(z\otr x)) \\
+\phi(x\utr y,z\otr y)
+\phi(x\utr y, (x\utr y)\utr(z\otr y)) 
&= & +\phi((z\otr x)\otr(y\otr x),z\otr x)
+\phi((z\otr x)\otr(y\otr x),(y\otr x)\utr(z\otr x))\\
+\phi((z\otr y)\otr (x\utr y),z\otr y) 
& &
+\phi(x\utr z,y\utr z)
+\phi(x\utr z,(x\utr z)\utr(y\utr z))\\
+\phi((z\otr y)\otr (x\utr y),(x\utr y)\utr(z\otr y))
& & +\phi((y\utr z)\otr(x\utr z),y\utr z)
+\phi((y\utr z)\otr(x\utr z),(x\utr z)\utr(y\utr z))\\
\end{array}$}\]
\end{itemize}
\end{definition}

\begin{example}
Let $X$ be the biquandle given by the operation tables
\[
\begin{array}{r|rr}
\utr & 1 & 2 \\\hline
1 & 2 & 2 \\
2 & 1 & 1
\end{array}
\quad 
\begin{array}{r|rr}
\otr & 1 & 2 \\\hline
1 & 2 & 2 \\
2 & 1 & 1
\end{array}.
\] 
Then $X$ has complete 2-fares over $A=\mathbb{Z}_5$ including
\[\begin{array}{r|rr}
\phi & 1 & 2 \\\hline
1 & 0 & 4 \\
2 & 4 & 2
\end{array}.\]
The Hopf link $L2a1$ has four $X$-colorings
\[\includegraphics{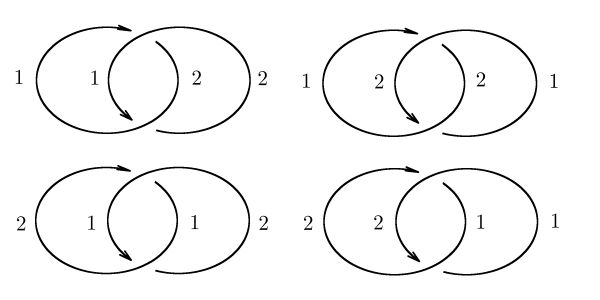}\]
and each coloring has a total of eight routes of order 2, resulting
in a multiset of four fares:
\[\begin{array}{rcl}
\Phi_{X}^{\phi,M} & = & \{\phi(1,1)+\phi(1,2)+\phi(2,1)+\phi(2,2)+\phi(1,1)+\phi(1,2)+\phi(2,1)+\phi(2,2), \\ & &
\phi(1,2)+\phi(1,2)+\phi(1,2)+\phi(1,2)+\phi(2,1)+\phi(2,1)+\phi(2,1)+\phi(2,1), \\ & &
\phi(2,1)+\phi(2,1)+\phi(2,1)+\phi(2,1)+\phi(1,2)+\phi(1,2)+\phi(1,2)+\phi(1,2),
\\ & & \phi(1,1)+\phi(1,2)+\phi(2,1)+\phi(2,2)+\phi(1,1)+\phi(1,2)+\phi(2,1)+\phi(2,2), \} \\
& = & \{0+4+4+2+0+4+4+2,\ 4+4+4+4+4+4+4+4, \\
&   & 4+4+4+4+4+4+4+4,\ 0+4+4+2+0+4+4+2\} \\
& = & \{0,2,2,0\}
\end{array}
\]
with polynomial versions $\Phi_X^{\phi,+}(L2a1)=2+x^2$ and 
$\Phi_X^{\phi,\times}(L2a1)=x^2(x-2)^2$.
\end{example}

\begin{example}
Let $X$ be the biquandle given by the operation tables
\[
\begin{array}{r|rrr}
\utr & 1 & 2 & 3 \\ \hline
1 & 2 & 2 & 1 \\
2 & 1 & 1 & 2 \\
3 & 3 & 3 & 3
\end{array}\quad
\begin{array}{r|rrr}
\otr & 1 & 2 & 3 \\ \hline
1 & 2 & 2 & 2 \\
2 & 1 & 1 & 1 \\
3 & 3 & 3 & 3
\end{array}.
\]
Then we compute via \texttt{python} that $X$ has complete 2-fares over 
$\mathbb{Z}_5$ including
\[\phi=\left[\begin{array}{rrr}
0 & 0 & 3 \\
0 & 0 & 3 \\
1 & 1 & 0
\end{array}\right]\]
(where the entry in row $j$ column $k$ is $\phi(j,k)$). 
Then the multiplicative polynomial fare values for the classical links with
up to seven crossings are given in the table.
\[
\begin{array}{r|l}
L & \Phi_X^{\phi,\times}(L) \\ \hline
L2a1 & x^5\\
L4a1 & x^5(x - 1)^4 \\
L5a1 & x^9\\
L6a1 & x^5(x - 1)^4\\
L6a2 & x^5 \\
L6a3 & x^5 \\
L6a4 & x^9(x - 1)^{18} \\
L6a5 & x^9(x - 1)^6 \\
L6n1 & x^9(x - 1)^6 \\
\end{array}\quad 
\begin{array}{r|l}
L & \Phi_X^{\phi,\times}(L) \\ \hline
L7a1 & x^9 \\
L7a2 & x^5(x - 1)^4 \\
L7a3 & x^9 \\
L7a4 & x^9 \\
L7a5 & x^5 \\
L7a6 & x^5 \\
L7a7 & x^{13}(x - 1)^2 \\
L7n1 & x^5(x - 4)^4 \\
L7n2 & x^9 \\
\end{array}
\]
\end{example}

Next we have another type of 2-fare.

\begin{definition}
Let $X$ be a finite biquandle and $A$ an abelian group. Then a \textit{through
biquandle fare} of order $n=2$ with coefficients in $A$ is a function
$\phi:X^2\to A$ satisfying the conditions
\begin{itemize}
\item[(i)] For all $x\in X$, 
\[\phi(x,x\utr x)+\phi(x\utr x,x)=0,\]
\item[(ii)] For all $x,y\in X$, 
\[\phi(x,x\utr y)+\phi(y\otr x,y)
-\phi(x\utr y,x)-\phi(y,y\otr x)=0\]
and
\item[(iii)] For all $x,y,z\in X$, 
\[\scalebox{0.95}{$\begin{array}{rcl}
\phi(x,x\utr y)
+\phi(y\otr x,y)
& & 
\phi(x,x\utr z)
+\phi(z\otr x,z)\\
+\phi(y,y\utr z)
+\phi(z\otr y,z)
&= & 
+\phi(y\otr x,(y\otr x)\utr(z\otr x)) +\phi((z\otr x)\otr(y\otr x),z\otr x) \\
+\phi(x\utr y, (x\utr y)\utr(z\otr y)) +\phi((z\otr y)\otr (x\utr y),z\otr y) 
%
& &
+\phi(x\utr z,(x\utr z)\utr(y\utr z))+\phi((y\utr z)\otr(x\utr z),y\utr z)\\
\end{array}$}\]
\end{itemize}
\end{definition}

\begin{example}
Let $X$ be the biquandle (indeed, quandle) defined by the operation tables
\[
\begin{array}{r|rrrr}
\utr & 1 & 2 & 3 & 4 \\ \hline
   1 & 1 & 3 & 4 & 2 \\ 
   2 & 4 & 2 & 1 & 3 \\
   3 & 2 & 4 & 3 & 1 \\
   4 & 3 & 1 & 2 & 4
\end{array}
\quad
\begin{array}{r|rrrr}
\otr & 1 & 2 & 3 & 4 \\ \hline
   1 & 1 & 1 & 1 & 1 \\ 
   2 & 2 & 2 & 2 & 2 \\
   3 & 3 & 3 & 3 & 3 \\
   4 & 4 & 4 & 4 & 4
\end{array}
\]
Our \textit{python} computations reveal 125 through 2-fares with $\mathbb{Z}_5$
coefficients including
\[
\begin{array}{r|rrrr}
\phi & 1 & 2 & 3 & 4 \\ \hline
1 & 0 & 0 & 3 & 0\\
2 & 0 & 0 & 3 & 0\\
3 & 2 & 2 & 0 & 2\\
4 & 0 & 0 & 3 & 0
\end{array}
\]
where the entry in row $j$ column $k$ is $\phi(j,k)$.

The figure eight knot $4_1$ has sixteen $X$-colorings including 
\[\includegraphics{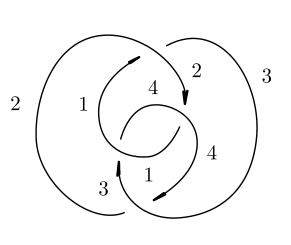}.\]
Each of the eight semi-arc starts a length-2 through path; summing the 
contributions, we obtain a total biquandle through fare of
\[\phi(2,2)-\phi(2,1)-\phi(1,1)+\phi(1,3)
+\phi(3,3)-\phi(3,4)-\phi(4,4)+\phi(4,2)=0-0-0+3+0-2-0+0=1.
\]
Repeating for the other 15 colorings, we obtain 
biquandle 2-fare multiset and polynomial invariant values
\[\Phi_X^{\phi,M}(4_1)=\{4\times 0, 6\times 1, 6\times 4 \},\
\Phi_X^{\phi,+}(4_1)=4+6x+6x^4\ \mathrm{and}\ 
\Phi_X^{\phi,\times}=x^4(x-1)^6(x-4)^4.\]
This example shows that these invariants are proper enhancements of the
(bi)quandle counting invariant since the trefoil knot $3_1$ has
the same counting invariant value of $16$ with respect to $X$
but has trivial values of the enhanced invariants,
\[\Phi_X^{\phi,M}=\{16\times 0 \},\
\Phi_X^{\phi,+}(0_1)=16\ \mathrm{and}\
\Phi_X^{\phi,\times}(0_1)=x^{16}.\]
\end{example}

\begin{example}
Let $X$ be the biquandle given by the operation tables
\[\begin{array}{r|rrrr}
\utr & 1 & 2 & 3 & 4\\ \hline
1 & 1 & 3 & 4 & 2 \\
2 & 2 & 4 & 3 & 1 \\
3 & 3 & 1 & 2 & 4 \\
4 & 4 & 2 & 1 & 3
\end{array}\quad
\begin{array}{r|rrrr}
\utr & 1 & 2 & 3 & 4 \\ \hline
1 & 1 & 1 & 1 & 1 \\
2 & 4 & 4 & 4 & 4 \\
3 & 2 & 2 & 2 & 3 \\
4 & 3 & 3 & 3 & 3
\end{array}.\]
Then we compute that $X$ has through fares over $\mathbb{Z}_6$ including
\[
\begin{array}{r|rrrr}
\phi & 1 & 2 & 3 & 4\\ \hline
1 & 3 & 1 & 4 & 1 \\
2 & 5 & 0 & 3 & 3 \\
3 & 2 & 3 & 0 & 0 \\
4 & 5 & 3 & 0 & 3
\end{array}
\]
and the multiplicative polynomial values for the prime classical knots
with up to eight crossings in the table.
\[
\begin{array}{r|l}
\Phi_X^{\phi,\times}(K) & K \\ \hline
x^4 & 5_1, 5_2, 6_1, 6_2, 6_3, 7_1, 7_4, 7_5, 7_6 7_7, 8_2, 8_3, 8_6, 8_7, 8_8, 8_9, 8_{12}, 8_{14}, 8_{16}, 8_{17} \\
x^4(x-2)^6(x-4)^6 & 4_1, 8_1, 8_{11} \\
x^{16} & 3_1, 7_2, 7_3, 8_4, 8_5, 8_{10}, 8_{13}, 8_{15}, 8_{19}, 8_{20}, 8_{21} \\
x^4(x-2)^{30}(x-4)^{30} & 8_{18}
\end{array}.
\]
In particular, we note that the invariant distinguishes several prime 
classical knots with equal counting invariant values and hence is a proper
enhancement.
\end{example}

Finally, we define a third type of 2-fare.

\begin{definition}
Let $X$ be a finite biquandle and $A$ an abelian group. Then a \textit{crooked
biquandle fare} of order $n=2$ with coefficients in $A$ is a function
$\phi:X^2\to A$ satisfying the conditions
\begin{itemize}
\item[(i)] For all $x\in X$, $\phi(x,x\utr x)+\phi(x\utr x,x)=0$,
\item[(ii)] For all $x,y\in X$, 
\[\phi(x,y)+\phi(y\otr x,x\utr y)
-\phi(y,x)-\phi(x\utr y,y\otr x)=0\]
and
\item[(iii)] For all $x,y,z\in X$, 
\[\begin{array}{rcl}
\phi(x,y)+\phi(y\otr x,x\utr y) 
& & \phi(x,z)+\phi(z\otr x,x\utr z)\\
+\phi(y,z)+\phi(z\otr y, y\utr z) 
&= & +\phi(y\otr x,z\otr x) 
+\phi((z\otr x)\otr(y\otr x),(y\otr x)\utr(z\otr x))\\
+\phi(x\utr y,z\otr y) 
+\phi((z\otr y)\otr (x\utr y),(x\utr y)\utr(z\otr y))
& & 
+\phi(x\utr z,y\utr z)
+\phi((y\utr z)\otr(x\utr z),(x\utr z)\utr(y\utr z))\\
\end{array}\]
\end{itemize}
\end{definition}

As with fares of type $n=1$, we have:

\begin{theorem}
Let $X$ be a biquandle, $L$ an oriented classical or virtual link, 
$A$ an abelian group and $\phi:X\to A$ a complete, through or crooked
biquandle $2$-fare. Then the multiset of fare values over the biquandle homset 
is invariant under Reidemeister moves.
\end{theorem}

\begin{proof}
The proof is a check of invariance under Reidemeister moves analogous to 
the proof of Theorem \ref{thm1}.
\end{proof}

For the case of fares of order $2$, a choice of through fare $\phi_1$
and crooked fare $\phi_2$ add to yield a complete fare $\phi_1+\phi_2$.
The converse is not true in general -- a complete fare need not decompose
as a sum of a through fare and a crooked fare. We will say a complete
2-fare is \textit{decomposable} if it is a sum of a crooked fare and a
through fare and \textit{indecomposable} otherwise.
\[\includegraphics{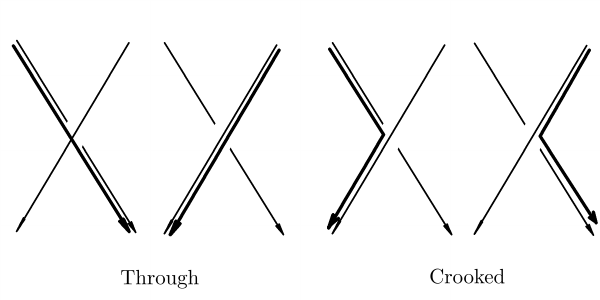}\]

\begin{example}
Let $X$ be the biquandle given by the operation tables
\[
\begin{array}{r|rrrr}
\utr & 1 & 2 & 3 & 4\\ \hline
   1 & 1 & 3 & 4 & 2\\ 
   2 & 2 & 4 & 1 & 3\\
   3 & 1 & 3 & 2 & 4\\
   4 & 4 & 2 & 3 & 1
\end{array}
\quad
\begin{array}{r|rrrr}
\otr & 1 & 2 & 3 & 4\\ \hline
   1 & 1 & 3 & 4 & 2\\ 
   2 & 2 & 4 & 1 & 3\\
   3 & 1 & 3 & 2 & 4\\
   4 & 4 & 2 & 3 & 1
\end{array}
\]

We use Python to compute that there are 3125 crooked biquandle 2-fares over 
$\mathbb{Z}_5$. One such example is 
\[
\begin{array}{r|rrrr}
\phi & 1 & 2 & 3 & 4\\ \hline
   1 & 1 & 3 & 4 & 2\\
   2 & 3 & 1 & 1 & 0\\
   3 & 0 & 2 & 4 & 2\\
   4 & 1 & 4 & 3 & 4
\end{array}
\]
where, as in Example 7, $\phi (j,k)$ is the entry in the $k$th column of 
the $j$th row.
There are 64 total X-colorings of the link L7a7, including the following:

\[\includegraphics{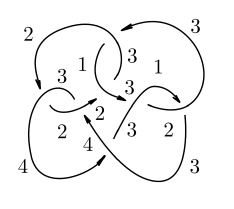}.\]

There are 7 crossings, and thus, 14 distinct crooked 2-routes in each 
$X$-coloring. Summing the fares over these, we obtain the crooked fare 
of this coloring to be:
\begin{eqnarray*}
\sum_{\mathrm{2-routes}} \phi(x,y) & = &
\phi(2,4) + \phi(3,2) +\phi(3,1) +\phi(3,2) +\phi(2,3)+ \phi(1,3) +\phi(4,2) 
 \\ & & 
+\phi(2,3) -\phi(3,1) -\phi(3,2) -\phi(1,3) -\phi(2,3) -\phi(4,4) -\phi(3,3) \\
& = & 4+1+4+1+2+0+0+2-4-1-0-2-4-4\\
& = & 4.
\end{eqnarray*}

Repeating for the other 63 colorings, we find that 12 have a fare of 4, 
24 have a fare of 2, and the remaining 36 have a fare of 0. Thus, we have 
the multiplicative polynomial invariant value 
\[\Phi_X^{\phi,\times}(L7a7)=x^{36}(x-2)^{24}(x-4)^{12}.\]

\end{example}

\section{\large\textbf{Questions}}\label{Q}

We conclude with some questions and directions for future research.

It is important first of all to note that the examples we have computed
here were chosen as small easily computable toy examples and should not 
be taken as representative of the power of this infinite family of 
invariants. Our examples use small cardinality biquandles with small 
cardinality finite ring coefficients; larger biquandles and larger finite
and infinite coefficient rings are expected to yield stronger invariants.

We have considered only the cases of 1-fares and 2-fares; the next 
cases to tackle should be 3-fares, 4-fares, et cetera. These cases pose
a technical challenge in that to find all possible cases we must expand the
neighborhood of each move to include more crossings. Nonetheless, $n$-fares
are definable for every natural number $n$.

Finally, we note that a crooked 2-fare is very similar to a 2-cocycle; this
inspires the question of what is the algebraic structure of the set of all
$n$-fares of a particular type (crooked, complete, through, etc.) for a 
given choice of biquandle and coefficient ring. Is there an underlying 
homology theory as with quandle/Yang-Baxter homology? What is the relation,
if any, with biquandle brackets?

\bigskip

\bibliographystyle{abbrv}
\bibliography{sn-ss2}

\noindent
\textsc{Department of Mathematical Sciences \\
Claremont McKenna College \\
850 Columbia Ave. \\
Claremont, CA 91711}

\end{document}